\documentclass{amsart}

\usepackage{amssymb,color}
\usepackage{amsfonts}
\usepackage{amsmath}
\usepackage{euscript}
\usepackage{enumerate}
\usepackage{pdfsync}
\newtheorem{theorem}{Theorem}[section]
\newtheorem{lemma}[theorem]{Lemma}

\newtheorem{prop}[theorem]{Proposition}
\newtheorem{cor}[theorem]{Corollary}

\newtheorem*{Theorem1'}{Theorem 1'}

\theoremstyle{definition}

\theoremstyle{remark}

\newcommand \GL{{\mathrm{GL}}}
\newcommand \Z{{\mathbb Z}}

\newcommand \N{{\mathbb N}}

\def\a{\alpha}
\def\b{\beta}

\begin{document}

\title[The automorphism group of certain polycyclic groups]{The automorphism group of certain polycyclic groups}

\author{Khalid Benabdallah}
\address{D\'epartement de math\'ematiques et de statistique, Universit\'e de Montr\'eal, Canada}
\email{khalid.benabdallah@umontreal.ca}

\author{Agust\'in D'Alessandro}
\address{Department of Mathematics and Statistics, University of Regina, Canada}
\email{dalessandro.ag@gmail.com}

\author{Fernando Szechtman}
\address{Department of Mathematics and Statistics, University of Regina, Canada}
\email{fernando.szechtman@gmail.com}
\thanks{The third author was partially supported by NSERC grant 2020-04062}





\begin{abstract} For $\b\in\Z$, let $G(\beta)=\langle A,B\,|\, A^{[A,B]}=A,\, B^{[B,A]}=B^\b\rangle$ be the
infinite Macdonald group, and set $C=[A,B]$. Then $G(\beta)$ is a nilpotent polycyclic group of the form 
$\langle A\rangle\ltimes\langle B,C\rangle$, where $A$ has infinite order.
If $\b\neq 1$, then $G(\beta)$ is of class 3 and $\langle B,C\rangle$ is a 
finite metacyclic group of order $|\b-1|^3$, which is an extension of $C_{(\b-1)^2}$ by  $C_{|\b-1|}$,
split except when $v_2(\b-1)=1$, while $G(1)$ is the integral Heisenberg group, of class 2 and  
$\langle B,C\rangle\cong\Z^2$.
We give a full description of the automorphism
group of $G(\b)$. If $\b\neq 1$, then $|\mathrm{Aut}(G(\b))|=2(\b-1)^4$ and we exhibit an imbedding 
$\mathrm{Aut}(G(\b))\hookrightarrow \GL_4(\Z/(\b-1)\Z)$, but for the case $\b\in\{-1,3\}$
when 5 is required instead of 4. When $\b$ is even the automorphism group 
of $\langle B,C\rangle$ can be obtained from the work of Bidwell and Curran \cite{BC}, 
and we indicate which of their automorphisms extend to an automorphism of $G(\b)$.
In general, we give necessary and sufficient conditions
for $G(\b)$ to be isomorphic to $G(\gamma)$. When $\gcd(\b-1,6)=1$,
we determine the automorphism group of $L(\b)=G(\b)/\langle A^{\b-1}\rangle$, which is a relative holomorph of $\langle B,C\rangle$,
and $\langle A^{\b-1}\rangle$ is a characteristic subgroup of $G(\b)$. 
The map $\mathrm{Aut}(G(\b))\to \mathrm{Aut}(L(\b))$ is injective and $\mathrm{Aut}(L(\b))$ 
is an extension of the Heisenberg group over $\Z/(\b-1)\Z$ direct product $C_{\b-1}$, by the holomorph
of $C_{\b-1}$.
\end{abstract}

\maketitle

\section{Introduction}

Recall that a group is metacyclic if it is an extension of a cyclic group by another cyclic group. The automorphism groups
of finite  metacyclic groups have been studied in detail. In 1970, Davitt showed that if $G$ is a finite
metacyclic $p$-group of order greater than $p^2$, with $p$ an odd prime, then $|G|$ is a factor of $|\mathrm{Aut}(G)|$.
This is related to the divisibility conjecture, which is explained in great detail by Dietz \cite{Di}.
The structure of the automorphism group of a finite split metacyclic $p$-group, with $p$ an odd prime, was determined
by Bidwell and Curran \cite{BC} in 2006. This was followed by the split case for $p=2$ by Curran \cite{C} in 2007, 
a case analyzed by a different method by Malinowska \cite{Ma} in 2008, as well as
the nonsplit case for $p$ odd by Curran \cite{C2} in 2008.
In 2008, Golasi\'nski and Gon{\c{c}}alves \cite{GG} studied the automorphism group of a split finite metacyclic group.
The case of a general finite metacyclic group was considered by Chen, Xiong, and Zhu \cite{CXZ} in 2018.

A group $G$ is polycyclic if it admits a series $\langle 1\rangle=G_0\unlhd G_1\unlhd\cdots \unlhd G_n=G$,
where each quotient group $G_{i+1}/G_i$ is cyclic. Polycyclic
groups are finitely presented (\cite[Section 2.2]{R}), which makes them interesting from a computational
point of view. In 1969, Auslander \cite{A} showed that the automorphism group of
a polycyclic group is also finitely presented.

In this paper, we describe the automorphism group of the polycyclic group
$G=\langle A\rangle\ltimes\langle B,C\rangle$,
where $\langle A\rangle$ is infinite and $\langle B,C\rangle$ is a finite characteristic 
metacyclic subgroup of~$G$. We thus have
a natural restriction homomorphism  $\mathrm{Aut}(G)\to \mathrm{Aut}(\langle B,C\rangle)$, connecting
the previously discussed case and the present one. The work on $\mathrm{Aut}(G)$ allows to determine, when $\gcd(\b-1,6)=1$,
the automorphism group of $G/\langle  A^{\b-1}\rangle$, which possesses more structural symmetry than $G$ itself.

The group $G$ is actually a member of a class of groups that have also received considerable attention,
namely groups that admit a finite presentation with as many generators and relations, with special interest
in finite groups of this type,  stemming from the first example of a finite group of this kind requiring 3 generators, namely
$M(a,b,c)=\langle x,y,z\,|\, x^y=x^a, y^{z}=y^b, z^x=z^c\rangle$, found in 1959 by Mennicke \cite{Me},
who proved that $M(a,b,c)$ is finite when $a=b=c\geq 2$. Other groups of this type were found by Macdonald \cite{M},
Wamsley \cite{W}, Post \cite{P}, Johnson \cite{J2}, Campbell and Robertson \cite{CR}, Abdolzadeh and Sabzchi \cite{AS},
and others. In most of these cases, the groups in question are defined in terms of certain integral parameters, and not
all choices of the given parameters result in a finite group. For instance, $M(-1,-1,-1)$ is infinite.

In 1962, Macdonald \cite{M} investigated the group
$$
G(\alpha,\beta)=\langle A,B\,|\, A^{[A,B]}=A^\a,\, B^{[B,A]}=B^\b\rangle,\quad \alpha,\beta\in\Z,
$$
which turns out to be finite provided $\a\neq 1$ and $\b\neq 1$. Macdonald made an incisive analysis of $G(\a,\b)$
under these assumptions. The full structure of $G(\a,\b)$ when $\a\neq 1$ and $\b\neq 1$ has recently been elucidated in \cite{S}.
Macdonald makes only a passing remark about the case when $\a=1$ or $\b=1$. Since, in general, $G(\alpha,\beta)\cong G(\b,\a)$,
one has $G(1,\beta)\cong G(\beta,1)$. In this paper we determine the structure of $G(\b)=G(1,\b)$ and use
this information to give detailed descriptions of $\mathrm{Aut}(G(\b))$ and $\mathrm{Aut}(G(\b)/\langle A^{\b-1}\rangle)$,
in the latter case when $\gcd(\b-1,6)=1$.
We set $C=[A,B]$ from now on, noting that $A^C=A$ and ${}^C B=B^\b$.

One readily sees that
$G(1)$ is isomorphic to the integral Heisenberg group $H(\Z)$ \cite[Chapter 5]{J}, whose automorphism group is an extension of
$\mathrm{Inn}(H(\Z))\cong\Z^2$ by $\mathrm{Out}(H(\Z))\cong\GL_2(\Z)$. Thus, we assume henceforth that $\b\neq 1$.

It turns out that $G(\b)=\langle A\rangle\ltimes\langle B,C\rangle$ is nilpotent of class 3, with torsion subgroup equal to 
$T(\b)=\langle B,C\rangle$, a metacyclic group
of order $|\b-1|^3$, which is an extension of $C_{(\b-1)^2}$ by 
$C_{|\b-1|}$, split except when $v_2(\b-1)=1$. In particular, $G(0)$ and $G(2)$ are trivial,
and if $\gamma\in\Z$, then $G(\b)\cong G(\gamma)$ is only possible when $\gamma=2-\b$, and we show that indeed 
$G(\b)\cong G(2-\b)$.

For each prime factor $p$ of $\b-1$, we may consider the subgroup
$G(\b)_p=\langle A\rangle\ltimes T(\b)_p$ of $G(\b)$, where $T(\b)_p$ stands for the Sylow $p$-subgroup of 
$T(\b)$, noting that $T(\b)_p$ is the torsion subgroup of $G(\b)_p$. If $v_p(\b-1)=m$, then 
$T(\b)$ is a metacyclic $p$-group of order $p^{3m}$. If $\gamma$ is also integer, it is then clear that $v_p(\gamma-1)=m$ is a necessary condition for $G(\gamma)_p$ to be isomorphic to to 
$G(\b)_p$, and we show that this condition is also sufficient.

Regarding the automorphism group of $G(\b)$, we show in Theorem \ref{auto} that
$$\mathrm{Aut}(G(\b))= \langle\Delta_1\rangle\ltimes (\mathrm{Inn}(G(\b))\times \langle\Delta_2\rangle)=
(\langle\Delta_1\rangle\ltimes \mathrm{Inn}(G(\b)))\times \langle\Delta_2\rangle,$$ 
where $\Delta_1$ has order 2, and $\Delta_2$ is a central automorphism of order $|\b-1|$ commuting with $\Delta_1$,
and $\mathrm{Inn}(G(\b))\cong H(\Z/(\b-1)\Z)$, so the order of $\mathrm{Aut}(G(\b))$ is always $2(\b-1)^4$.
When $\b\not\in\{-1,3\}$, we can view $\mathrm{Aut}(G(\b))$ as the subgroup of $\GL_4(\Z/(\b-1)\Z)$ of all matrices of the form
\begin{equation}\label{mat}
\begin{pmatrix} 1 & a & c & d\\ 0 & s & b & 0\\ 0 & 0 & 1 & 0\\0 & 0 & 0 & 1\end{pmatrix},
\end{equation}
where $s\in\{-1,1\}$. Furthermore, we study the restriction map
$\Lambda:\mathrm{Aut}(G(\b))\to \mathrm{Aut}(T(\b))$ and determine its kernel and image.
When $\b$ is even, the automorphism group 
of $T(\b)$ can be derived from \cite{BC}, 
and we indicate which of their automorphisms extend to an automorphism of $G(\b)$.

If $S$ is a group, its holomorph $\mathrm{Hol}(S)=\mathrm{Aut}(S)\ltimes S$, under the natural action,
and given a subgroup $X$ of $\mathrm{Aut}(S)$, the relative holomorph $\mathrm{Hol}(X,S)=X\ltimes S$ as a subgroup
of $\mathrm{Hol}(S)$. Suppose next $\gcd(\b-1,6)=1$ and let $K=\langle A^{\b-1}\rangle$.
Then $K$ is a characteristic subgroup of $G(\b)$, and $L(\b)=G(\b)/K\cong 
\mathrm{Hol}(\langle \lambda\rangle,\langle B,C\rangle)$, where $B^\lambda=BC^{-1}$
and $C^\lambda=C$. Moreover, the natural map $\mathrm{Aut}(G(\b))\to \mathrm{Aut}(L(\b))$ is injective and 
$\mathrm{Aut}(L(\b))$ is an extension of $H(\Z/(\b-1)\Z)\times C_{\b-1}$ by $\mathrm{Hol}(C_{\b-1})$, so
$|\mathrm{Aut}(L(\b))|=\varphi(|\b-1|)|\b-1|^5$, where $\varphi$ is Euler's totient function.

In terms of notation, function composition proceeds from left to right. Given a group $S$, we~set
$$
[a,b]=a^{-1}b^{-1}ab,\; b^a=a^{-1}ba,\; {}^a b=aba^{-1},\quad a,b\in S.
$$

We let $\delta:S\to\mathrm{Aut}(S)$ stand for the canonical map $a\mapsto a\delta$, where $a\delta$ is conjugation by $a$,
namely the map $b\mapsto b^a$. Moreover, we let $\langle 1\rangle=Z_0(S),Z(S)=Z_1(S),Z_2(S),\dots$ stand for the terms of the upper central series of~$T$, so that 
$Z_{i+1}(S)/Z_i(S)$ is the center of $S/Z_i(S)$, and write 
$\mathrm{Aut}_i(S)$ for the kernel of the canonical map
$\mathrm{Aut}(S)\to \mathrm{Aut}(S/Z_i(S))$, $i\geq 0$. 
The automorphisms of $S$ belonging to $\mathrm{Aut}_1(S)$ are said to be central.
Note that central
and inner automorphisms commute with each other. 
We further let
$S=\gamma_1(S),\gamma_2(S),\gamma_3(S),\dots$ stand for the terms of the lower central series of~$S$, so that 
$\gamma_{i+1}(S)=[S,\gamma_i(S)]$.

\section{Preliminary Observations}\label{po}

We write $G=G(\b)$, $T=T(\b)$, $Z_i(G)=Z_i$, and $\gamma_i(G)=\gamma_i$ from now on, unless confusion is possible.

$\bullet$ $A$ has infinite order. Let $\langle x\rangle$ be an infinite cyclic group. Then the assignment $A\mapsto x$, $B\mapsto 1$,
extends to a group epimorphism $G\to \langle x\rangle$.

$\bullet$ $G$ is nilpotent of class at most 3 and $B^{\b-1}\in Z$. Conjugating the relation $A^B=AC$ by~$C^{-1}$,
we obtain $A^{B^\b}=AC=A^B$, so $B^\b B^{-1}$ commutes with $A$. 
Now ${}^C B=B^\b=B B^{\b-1}$, so $C$ commutes with $B$ modulo $Z$, and hence $C\in Z_2$.
Since $[A,B]\in Z_2$, we see that $G/Z_2$
is abelian, so $Z_3=G$. Appealing to \cite[5.1.7]{R}, we readily see that $\gamma_2=\langle C,B^{\b-1}\rangle$, $\gamma_3=\langle B^{\b-1}\rangle$, $\gamma_4=\langle 1\rangle$.

$\bullet$ $B^{(\b-1)^2}=1$. As $B^{\b-1}\in Z$,  we have $B^{\b-1}={}^C (B^{\b-1})=B^{\b(\b-1)}$, so $B^{(\b-1)^2}=1$.

$\bullet$ If $\b=2$ or $\b=0$, then $B=1$, whence $G=\langle A\rangle\cong\Z$. Thus, we assume for the remainder of the paper that
$\b>2$ or $\b<0$

$\bullet$ If $i>0$, then \cite[Eq. (1.6)]{M}
$$
(B^i)^A=(B^A)^i=(BC^{-1})^i=C^{-i} B^{\b(1+\b+\cdots+\b^{i-1})}.
$$

$\bullet$ If $2\nmid (\b-1)$ then $C^{\b-1}=1$, while if $2\mid (\b-1)$ then $C^{\b-1}=B^{(\b-1)^2/2}$. Setting $i=|\b-1|$,
and taking into account $B^i\in Z$ and $B^{(\b-1)^2}=1$, we deduce
$$
B^i=(B^i)^A= C^{-i} B^{\b(\b^i-1)/(\b-1)}=C^{-i} B^{i+\binom{i}{2}(\b-1)}.
$$


$\bullet$ $\langle B\rangle$ is normal in $\langle B,C\rangle$, which is a finite normal subgroup of $G$ whose order divides $(\b-1)^3$.
As $\langle B\rangle$ is finite and $B^C=B^\b$, it follows that $\langle B\rangle$ is normal 
in $\langle B,C\rangle$. Since $B^A=BC^{-1}$ and $C^A=C$, we see that $\langle B,C\rangle$ is normal in $G$. 
As $\langle B\rangle$ is normal in $\langle B,C\rangle$,
we have $\langle B,C\rangle=\langle B\rangle\langle C\rangle$. As the order of $B$ is a factor of $(\b-1)^2$ 
and the order of $C$ modulo $\langle B\rangle$ is a factor of $\b-1$, it follows that $|\langle B,C\rangle|$ is a factor of $(\b-1)^3$.

$\bullet$  $G=\langle A\rangle\ltimes \langle B,C\rangle$. Since $\langle B,C\rangle$ is normal in $G$, 
we see that $\langle A\rangle\langle B,C\rangle$ is a subgroup of $G$. As it contains  $A$ and $B$, it must equal $G$. Since
every element of $\langle B,C\rangle$ has finite order, and the only element of finite order of $\langle A\rangle$ is 1,
then $\langle B,C\rangle$ intersects $\langle A\rangle$ trivially.

Recall  \cite[5.2.7]{R} that the torsion elements of a nilpotent group form a subgroup. Since $\langle B,C\rangle$ is part of the
torsion subgroup of $G=\langle A\rangle\ltimes \langle B,C\rangle$, and $\langle A\rangle$ is torsion-free, we deduce

$\bullet$ The torsion subgroup of $G$ is $T=\langle B,C\rangle$. 


$\bullet$ For each prime factor $p\in\N$ of $\b-1$, let $T_p$ stand for the Sylow $p$-subgroup of $T$. Then $T$
is the direct product of the $T_p$, as $p$ runs through all prime factors $p\in\N$ of $\b-1$.
This follows from the fact that $T$ is a finite nilpotent group.

\section{Structure of the torsion subgroup of $G$}

In this section we determine the structure of the torsion subgroup $T$ of $G$.

\begin{prop}\label{gam} We have $G(2-\b)\cong G(\b)$.
\end{prop}

\begin{proof} Set
$G(2-\beta)=\langle X,Y\,|\, X^{[X,Y]}=X,\, Y^{[Y,X]}=Y^{2-\b}\rangle$. Since 
$\b(2-\b)\equiv 1\mod (\b-1)^2$, it follows that the assignment $A\mapsto  X^{-1}$, $B\mapsto Y$, extends
to an isomorphism $G(\b)\to G(2-\b)$ with inverse given by $X\mapsto  A^{-1}$, $Y\mapsto B$.
\end{proof}

In view of Proposition \ref{gam}, we may assume from now on that $\b>2$, and we will do so.

\begin{prop}\label{str1} Suppose that $\b$ is even. Then $T=\langle C\rangle\ltimes \langle B\rangle\cong C_{\b-1}\ltimes C_{(\b-1)^2}$.
\end{prop}

\begin{proof} We will construct a group epimorphism $f:G\to U$, where $U=\langle z\rangle\ltimes (\langle y\rangle\ltimes \langle x\rangle)$, 
$y$ has order $\b-1$, and $x$ has order $(\b-1)^2$, in such a way that $T^f=\langle x,y\rangle$. 
Since $|T|\leq (\b-1)^3$, it will follow that $f$ restricts to an isomorphism $T\to 
\langle x,y\rangle$.

We begin with the group $E=\langle y\rangle\ltimes \langle x\rangle$, where $x$ has order $(\b-1)^2$, $y$ has order $\b-1$, and
$$
{}^y x=x^\b.
$$
This gives an action of $\langle y\rangle$ on $\langle x\rangle$ by automorphisms, since $(\b-1)^2$ is a factor of $\b^{\b-1}-1$,
so $y^{\b-1}$ does act trivially on $x$. The defining relations of $E=\langle x,y\rangle$ are:
$$
x^{(\b-1)^2}=1,y^{\b-1}=1, {}^y x=x^\b.
$$
We next consider an infinite cyclic group $\langle z\rangle$. We claim that
$$
x^z=xy^{-1}, y^z=y,
$$
extends to an action of $\langle z\rangle$ on $E$ by automorphisms. We need to verify that the defining relations of $E$
are preserved by the action of $z$, as the resulting endomorphism will then clearly be an epimorphism and hence an automorphism of $E$.
For this purpose, set $\b_0=2-\b$, recalling that $\b\b_0\equiv 1\mod (\b-1)^2$.

The relation $y^{\b-1}=1$ is obviously preserved. The preservation of $x^{(\b-1)^2}=1$ means that $(xy^{-1})^{(\b-1)^2}=1$.
This is true, since
$$
(xy^{-1})^{(\b-1)^2}=x^{1+\b_0+\cdots+\b_0^{(\b-1)^2-1}}y^{-(\b-1)^2}=1,
$$
given that $(\b-1)^2$ is a factor of $(\b_0^{(\b-1)^2}-1)/(\b_0-1)$. 
The preservation of ${}^y x=x^\b$ means that ${}^y (xy^{-1})=(xy^{-1})^\b$.
Here ${}^y (xy^{-1})=x^\b y^{-1}$. On the other hand, since $\b\equiv 1\mod \b-1$,
$$
(xy^{-1})^{\b}=x^{1+\b_0+\cdots+\b_0^{\b-1}}y^{-\b}=x^{\b}y^{-1},
$$
because $\b$ is even, which implies
$$
(\b_0^\b-1)/(\b_0-1)\equiv \b\mod (\b-1)^2.
$$

To define $f$, simply let $A\mapsto z$ and $B\mapsto x$. We need to verify that the defining relations of $G$
are preserved. Now $x^z=xy^{-1}$ gives $[x,z]=y^{-1}$, so $[z,x]=y$, and therefore $z^{[z,x]}=z^y=z$. Moreover,
${}^{[z,x]} x= {}^y x=x^\b$. This completes the proof.
\end{proof}

\begin{prop}\label{str2} Suppose that $\b$ is odd. Then $T$ is a split extension of $C_{(\b-1)^2}$ by
$C_{(\b-1)}$,  except when $v_2(\b-1)=1$, in which case the Sylow 2-subgroup of $T$ is $Q_8$.
\end{prop}

\begin{proof} We start with a cyclic group $\langle x\rangle$ of order $(\b-1)^{2}$ and use the theory of group extensions with cyclic quotient,
as explained in \cite[Chapter III, Section 7]{Z}), to construct a group $\langle x,y\rangle$
such that: $\langle x\rangle$ is normal in $\langle x,y\rangle$ with cyclic quotient of order $\b-1$ and, in fact, 
$y^{(\b-1)}=x^{(\b-1)^2/2}$,
with ${}^y x=x^\b$. Note that $\langle x,y\rangle$ has order $(\b-1)^{3}$ and set $\b_0=2-\b$.

To achieve this consider the automorphism, say $\Omega$, of $\langle x\rangle$, given by $x\mapsto x^{\b_0}$.
This fixes $x^{(\b-1)^2/2}$, since $\b_0(\b-1)^2/2\equiv (\b-1)^2/2\mod (\b-1)^2$. Moreover, $\Omega^{(\b-1)}$ is conjugation
by $x^{(\b-1)^2/2}$, namely trivial, since 
$$
\b_0^{\b-1}\equiv 1\mod (\b-1)^2.
$$
This produces the required group. The defining relations of $\langle x,y\rangle$ are:
$$
y^{\b-1}=x^{(\b-1)^2/2}, {}^y x=x^\b, y^{2(\b-1)}=1.
$$

We next define an automorphism $\Delta$ of $\langle x,y\rangle$ such that
$$
x\mapsto xy^{-1}, y\mapsto y.
$$
To see that this assignment does extend to an endomorphism, and hence an automorphism, we need to verify that
the above defining relations are preserved. The case of $y^{2(\b-1)}=1$ is obvious. As for 
$y^{\b-1}=x^{(\b-1)^2/2}$ we need to show that $y^{\b-1}=(xy^{-1})^{(\b-1)^2/2}$. We have
$$
(xy^{-1})^{(\b-1)^2/2}=x^{1+\b_0+\cdots+\b_0^{(\b-1)^2/2}-1}y^{-(\b-1)^2/2},
$$
where
$$
(\b_0^{(\b-1)^2/2}-1)/(\b_0-1)\equiv [1+s(\b-1)/2](\b-1)^2/2\mod (\b-1)^2,
$$
where $s$ is odd. Set $m=v_2(\b-1)$. If $m\geq 2$, then $s(\b-1)/2$ is even and $-(\b-1)^2/2\equiv 0\mod 2(\b-1)$.
Thus $(xy^{-1})^{(\b-1)^2/2}=x^{(\b-1)^2/2}=y^{\b-1}$ when $m\geq 2$.
Suppose next $m=1$. Then $1+s(\b-1)/2$ is even and $-(\b-1)^2/2\equiv \b-1\mod 2(\b-1)$.
Thus $(xy^{-1})^{(\b-1)^2/2}=y^{\b-1}$ is true when $m=1$.

Regarding ${}^y x=x^\b$, we need to show that ${}^y (xy^{-1})=(xy^{-1})^\b$.
Here ${}^y (xy^{-1})=x^\b y^{-1}$, while
$$
(xy^{-1})^{\b}=x^{1+\b_0+\cdots+\b_0^{\b-1}}y^{-\b}.
$$
Note that $y^{-\b}=y^{-(\b-1)}y^{-1}=x^{(b-1)^2/2}y^{-1}$. Moreover,
$$
(\b_0^\b-1)/(\b_0-1)\equiv \b+r(\b-1)^2/2\equiv \b+(\b-1)^2/2\mod (\b-1)^2,
$$
where $r$ is odd. Since $x^{(\b-1)^2/2}x^{(\b-1)^2/2}=1$, we see that ${}^y (xy^{-1})=(xy^{-1})^\b$. 

This produces a group $U=\langle z\rangle\ltimes (\langle x,y\rangle)$ and, as before, an epimorphism $f:G\to U$ such that 
$T^f=\langle x,y\rangle$, which ensures that $T\cong \langle x,y\rangle$.

Let $m=v_2(\b-1)$. If $m>1$, then $T=\langle B^{(\b-1)/2}C\rangle\ltimes  \langle B\rangle$.
Suppose $m=1$. Then the Sylow 2-subgroup $T_2$ of $T$ has order 8, and the projections $b$ and $c$
of $B$ and $C$ onto $T_2$ satisfy $b^4=1$, $b^2=c^2$, and $b^c=b^{-1}$, so $T_2\cong Q_8$.
As $T$ is the direct product of its Sylow subgroups and $T_2$ does not split, neither does $T$.
\end{proof}

\begin{cor} The nilpotency class of $G$ is 3.
\end{cor}

\begin{proof} We already noted that $\gamma_3=\langle B^{\b-1}\rangle$, which is a non-trivial central subgroup of $G$ by Propositions
~\ref{str1} and \ref{str2}.
\end{proof}

\begin{cor} 
We have $\langle B\rangle\cap \langle C\rangle=\langle C^{\b-1}\rangle$.
\end{cor}

\begin{proof} Immediate consequence of Propositions
\ref{str1} and \ref{str2}.
\end{proof}

\begin{cor} 
For $\gamma\in\Z$, we have $G(\gamma)\cong G(\b)$ if and only if $\gamma=\b$ or $\gamma=2-\b$.
\end{cor}

\begin{proof} Suppose $G(\gamma)\cong G(\b)$. Since the torsion subgroups of $G(\b)$ and $G(\gamma)$ have orders $(\b-1)^3$
and $|\gamma-1|^3$, it follows that $\gamma=\b$ or $\gamma=2-\b$. 
\end{proof}

\begin{cor} 
The first power of $A$ in $Z$ is $\b-1$ if $\b$ is even and $2(\b-1)$ if $\b$ is odd.
\end{cor}

\begin{proof} We have $B^{A^n}=BC^{-n}$, $n\in\N$, and by Propositions
\ref{str1} and \ref{str2}, the order of $C$ is $\b-1$ if $\b$ is even and $2(\b-1)$ if $\b$ is odd.
\end{proof}

\begin{cor} 
The first power of $B$ in $Z$ is $\b-1$. 
\end{cor}

\begin{proof} By Propositions
\ref{str1} and \ref{str2}, the order of $B$ is $(\b-1)^2$, and $[B^n,C^{-1}]=B^{n(\b-1)}$, $n\in\N$.
\end{proof}

\begin{prop}\label{centro} Set $W=\langle A^{\b-1}, B^{\b-1}\rangle$ if $\b$ is even,
$W=\langle B^{\b-1}, A^{\b-1} C^{(\b-1)/2}\rangle$ if $\b$ is odd. Then
$$
W=Z=C_G(\langle B,C\rangle),\; G/Z\cong H(\Z/(b-1)\Z),\; |G/Z|=(\b-1)^3.
$$
\end{prop}

\begin{proof} It is easy to see that $W\subseteq Z\subseteq C_G(\langle B,C\rangle)$, although a calculation
is required to verify that $[B,A^{\b-1} C^{(\b-1)/2}]=1$ when $\b$ is odd, which is essentially done below.

We proceed to show that  $C_G(\langle B,C\rangle)\subseteq W$. Suppose that $A^i C^k B^j\in C_G(\langle B,C\rangle)$.
Then $A^i C^k B^j$ commutes with $C$, and therefore so does $B^j$, whence $j\equiv 0\mod \b-1$, which forces
$B^j\in W$ and $A^i C^k\in C_G(B)$. We are reduced to show that $A^i C^k\in W$, and we assume without loss that $k>0$. 
We have $B^{A^i}=B^{C^{-k}}$, where
$B^{A^i}=BC^{-i}$ and $B^{C^{-k}}={}^{C^k} B=B^{\b^{k}}$, so $BC^{-i}=B^{\b^{k}}$, whence $i\equiv 0\mod \b-1$.
Suppose first $\b$ is even. Then $C^i=1=B^{\b^{k}-1}$, which forces 
$k\equiv 0\mod \b-1$, so $C^k=1$, $A^i C^k=A^i\in W$.
Suppose next $\b$ is odd. Two cases arise: $i\equiv 0\mod 2(\b-1)$; $i\not\equiv 0\mod \b-1$. In the first case
$C^{-i}=1$, $A^i\in W$, $B^{\b^{k}-1}=1$, $k\equiv 0\mod \b-1$, $C^k\in\langle B^{(\b-1)^2/2}\rangle
\subseteq \langle B^{\b-1}\rangle\subseteq W$. In the second case, $\b^{k}\equiv 1\mod (\b-1)^2/2$ but
$\b^{k}\not\equiv 1\mod (\b-1)^2$. This forces $k\equiv 0\mod (b-1)/2$ but $k\not\equiv 0\mod b-1$. Thus
$A^{i}C^k=A^{r(\b-1)}C^{s(\b-1)/2}$, where $r$ and $s$ are odd, so $A^{r(\b-1)}C^{s(\b-1)/2}=A^{\b-1}C^{\b-1)/2}h$, where
$h\in W$ and $A^{\b-1}C^{(\b-1)/2}\in W$, and therefore $A^{i}C^k\in W$.

We next show that $G/Z\cong H(\Z/(b-1)\Z)$. If $\b$ is even, then $Z=\langle A^{\b-1}, B^{\b-1}, C^{\b-1}\rangle$
and the result is clear in this case. Suppose next that $\b$ is odd. Then
$$
G/\langle A^{2(\b-1)}, B^{\b-1}, C^{\b-1}\rangle\cong C_{2(\b-1)}\ltimes (C_{\b-1}\times C_{\b-1}),$$
so
$$
G/Z\cong (G/\langle A^{2(\b-1)}, B^{\b-1}, C^{\b-1}\rangle)/(Z/\langle A^{2(\b-1)}, B^{\b-1}, C^{\b-1}\rangle),
$$
where
$$
Z/\langle A^{2(\b-1)}, B^{\b-1}, C^{\b-1}\rangle\cong C_2.
$$
Thus $|G/Z|=|\b-1|^3$. We claim that $G/Z\cong H(\Z/(b-1)\Z)$. Indeed, let $x,y,z$ be the standard generators of $H(\Z/(b-1)\Z)$,
subject to the defining relations:
$$
[x,y]=z, [x,z]=1=[y,z], x^{\b-1}=y^{\b-1}=z^{\b-1}=1.
$$
Note that
$$
(xy)^n=x^n y^n c^{-n(n-1)/2},\quad n\in\N.
$$
In particular, $(xy)^{\b-1}=z^{-\b(\b-1)/2}=z^{-(\b-1)/2}$. Consider the assignment $A\mapsto xy$, $B\mapsto y$.
Since $[xy,y]=[x,y]^y\, [y,y]=z$, we have
$$
(xy)^{[xy,y]}=xy,\; {}^{[xy,y]} y= y=y^\b.
$$
Thus, the given assignment extends to a group epimorphism $h:G\to H(\Z/(b-1)\Z)$. From $(xy)^{\b-1}=z^{-(\b-1)/2}$,
we deduce that $Z$ is included in the kernel of $h$. But $|G/Z|=(\b-1)^3=|H(\Z/(b-1)\Z)|$, so $Z=\ker(h)$ and $G/Z\cong H(\Z/(b-1)\Z)$.

\end{proof}

\section{The automorphism group of $G$}

\begin{theorem}\label{auto} The assignments $A\mapsto A^{-1}$, $B\mapsto B^{-1}$ and
$A\mapsto A B^{\b-1}$, $B\mapsto B$ extend to commuting automorphisms of $G$, say $\Delta_1$ and $\Delta_2$,
of orders 2 and $|\b-1|$, respectively.
Moreover, $\mathrm{Aut}(G)=(\langle \Delta_1\rangle\ltimes \mathrm{Inn}(G))\times\langle \Delta_2\rangle$
has order $2(\b-1)^4$, where $(A\delta)^{\Delta_1}=A^{-1}\delta$ and $(B\delta)^{\Delta_1}=B^{-1}\delta$.
If $\b\notin\{-1,3\}$, then $\mathrm{Aut}(G)$ is isomorphic to the subgroup of $\GL_4(\Z/(\b-1)\Z)$ of all matrices of the form
{\rm (\ref{mat})}.
\end{theorem}

\begin{proof} Let $\Gamma\in \mathrm{Aut}(G)$.

\medskip

\noindent{\sc Step 1.} The assignment $A\mapsto A^{-1}$, $B\mapsto B^{-1}$ extends to an automorphism $\Delta_1$ of $G$.

\medskip 

As $A$ and $B$ commute with $C$ modulo $Z$, we have $[A^{-1},B^{-1}]\equiv C\mod Z(G)$.

\medskip 

\noindent{\sc Step 2.} The assignment $A\mapsto A B^{\b-1}$, $B\mapsto B$ extends to a central automorphism $\Delta_2$ of $G$
of order $\b-1$ that commutes with $\Delta_1$.

\medskip 

This is clear.

\medskip 

\noindent{\sc Step 3.} We have $A^{\Gamma_2}=A B^i C^j$, where $\Gamma_2=\Gamma\Delta_1^q$ for a suitable $q$.

\medskip

As the torsion subgroup of $G$, the subgroup $\langle B,C\rangle$
is characteristic in $G$. Since $A$ must belong to the image $\Gamma$, it follows that $A^\Gamma=A^{\pm 1} B^i C^j$.

\medskip

\noindent{\sc Step 4.} We have $A^{\Gamma_3}=A B^i$, where $\Gamma_3=\Gamma_2 (B\delta)^f$ for a suitable $f$.

\medskip

By Step 3, $A^{\Gamma_2}=A B^i C^{-f}$, where $f>0$. Now
\begin{equation}\label{con}
A^{B^f}=A B^{(\b-1)(\b+2\b^2+\cdots+(f-1)\b^{f-1})} C^f,
\end{equation}
$$
(C^{-f})^{B^f}= C^{-f} B^{f (1-\b^f)}.
$$
Then $\Gamma_2(B\delta)^f$ sends $A$ to $AB^t$ for some $t\in\Z$.

\medskip

\noindent{\sc Step 5.} The centralizer of $A$ in $G$ is $\langle A,C, B^{\b-1}\rangle$. 

\medskip

It is clear that the stated group is contained in centralizer of $A$, so it suffices to show that $A$ has at least $\b-1$ conjugates
in $G$. This follows from (\ref{con}) and 
$
\langle B\rangle\cap \langle C\rangle=\langle C^{\b-1}\rangle.
$

\medskip

\noindent{\sc Step 6.} In $A^{\Gamma_3}=A B^i$, we have $i\equiv 0\mod \b-1$. 

\medskip

This is so because, by (\ref{con}),
$$
(AB^i)^{B^f}=A B^{(\b-1)(\b+2\b^2+\cdots+(f-1)\b^{f-1})} C^f B^i=A B^{(\b-1)(\b+2\b^2+\cdots+(f-1)\b^{f-1})}B^{i\b^f} C^f.
$$
Since, $\langle B\rangle\cap \langle C\rangle=\langle C^{\b-1}\rangle$, it follows that $AB^i$ has exactly $\b-1$
conjugates by powers of $B$. As $A$ has $\b-1$
conjugates in $G$ and $A^{\Gamma_3}=AB^i$, the above are all the conjugates of $AB^i$ in $G$.  But
$$
{}^C (AB^i)= A B^{i\b},
$$
so $A B^{i\b}$ must be equal to one of the $(AB^i)^{B^f}$, with $0<f\leq \b-1$. From 
$\langle B\rangle\cap \langle C\rangle=\langle C^{\b-1}\rangle$,
the only possibility is $A B^{i\b}=AB^i$, whence $i(\b-1)\equiv 0\mod (\b-1)^2$, as claimed.

\medskip

\noindent{\sc Step 7.} We have $A^{\Gamma_4}=A$, where $\Gamma_4=\Gamma_3\Delta_2^q$ for a suitable $q$.

\medskip

This follows from Steps 2 and 6.

\bigskip

As $\langle B,C\rangle$ is a characteristic subgroup of $G$, we have $B^{\Gamma_4}=B^i C^j$, $C^{\Gamma_4}=B^u C^v$ for some $i,j,u,v\in\N$.

\bigskip

\noindent{\sc Step 8.} In $C^{\Gamma_4}=B^u C^v$, we have $u\equiv 0\mod \b-1$.

\medskip

Since $[A,C]=1$, we see that $A$ commutes with $B^u C^v$, and hence with $B^u$. It follows from (\ref{con}) that
$C^u\in\langle B\rangle$, so $u\equiv 0\mod \b-1$.

\medskip

\noindent{\sc Step 9.} In $B^{\Gamma_4}=B^i C^j$, we have $\gcd(i,\b-1)=1$. 

\medskip

This follows from the fact $\Gamma_4$ restricts to an automorphism of $\langle B,C\rangle$, with 
$C^{\Gamma_4}=B^u C^v$ and $u\equiv 0\mod \b-1$.

\medskip

\noindent{\sc Step 10.} We have $B^{\Gamma_5}=B^i$ and $A^{\Gamma_5}=A$, where $\Gamma_5=\Gamma_4 (A\delta)^q$ for a suitable $q$.

\medskip

For $x\in\N$, we have
$$
(B^i)^{A^x}=(B^{A^x})^i=(BC^{-x})^i=C^{-xi} B^{\b^x(1+\b^x+\cdots+\b^{x(i-1)})}.
$$
Then $\Gamma_4(A\delta)^x$ still fixes $A$ and sends $B$ to
$$
(B^i C^j)^{A^x}=  C^{-xi} B^{\b^x(1+\b^x+\cdots+\b^{x(i-1)})} C^j=B^t C^{-xi} C^j.
$$
Since $\gcd(i,\b-1)=1$, we may solve the congruence $x i\equiv j\mod \b-1$.

\medskip

\noindent{\sc Step 11.} In $C^{\Gamma_5}=B^u C^v$, with $u\equiv 0\mod \b-1$, we also have $v\equiv 1\mod \b-1$,
so $C^{\Gamma_5}=B^w C$.

\medskip

Since ${}^C B=B^\b$, we must have 
$$B^{i \b^v}={}^{C^v}\! B^i={}^{B^u C^v}\! B^i=B^{i\b}.
$$
As $\gcd(i,\b-1)=1$, we infer $\b^v\equiv \b\mod(\b-1)^2$. Here $\b^v\equiv 1+v(\b-1)\mod(\b-1)^2$ and $\b\equiv 1+(\b-1)\mod(\b-1)^2$,
whence $v(\b-1)\equiv\b-1\mod  (\b-1)^2$, so $v\equiv 1\mod \b-1$.

\medskip

\noindent{\sc Step 12.} In $B^{\Gamma_5}=B^i$ and $A^{\Gamma_5}=A$, we have $i\equiv 1\mod \b-1$.

\medskip

From $[A,B]=C$ and Step 11, we infer $[A,B^i]=B^w C$, that is
$$
B^{(\b-1)(\b+2\b^2+\cdots+(i-1)\b^{i-1})} C^i=B^w C,
$$
so $\langle B\rangle\cap \langle C\rangle=\langle C^{\b-1}\rangle$ forces $i\equiv 1\mod \b-1$.

\medskip

\noindent{\sc Step 13.} We have $A^{\Gamma_6}=A$ and $B^{\Gamma_6}=B$, where 
$\Gamma_6=\Gamma_5(C\delta)^q$ for a suitable $q\in\Z$, so that $\Gamma\in \langle \Delta_1,\Delta_2\rangle\mathrm{Inn}(G)$.

\medskip

By Step 12, we can find a suitable $q$ so that $\Gamma_6$ is the identity automorphism. The previous steps now ensure
that $\Gamma\in \langle \Delta_1,\Delta_2\rangle\mathrm{Inn}(G)$.

\medskip

\noindent{\sc Step 14.} The group $\mathrm{Inn}(G)\langle \Delta_2\rangle=\mathrm{Inn}(G)\times\langle \Delta_2\rangle$
has order $(\b-1)^4$.

\medskip

As mentioned in the Introduction, central and inner automorphisms commute with each other.
Suppose $(A^s C^r B^f)\delta=\Delta_2^q$, where $f>0$. Then 
$A^{B^f}=A B^{(\b-1)(\b+2\b^2+\cdots+(f-1)\b^{f-1})} C^f$ is equal to
$A B^{q(\b-1)}$, which forces $f\equiv 0\mod \b-1$, and fortiori $q\equiv 0\mod \b-1$. Hence $\mathrm{Inn}(G)\cap\langle \Delta_2\rangle$
is trivial, so $|\mathrm{Inn}(G)\times\langle\Delta_2\rangle|=(\b-1)^4$ by Proposition \ref{centro}.

\medskip

\noindent{\sc Step 15.} We have $\langle \Delta_1\rangle\cap(\mathrm{Inn}(G)\times\langle \Delta_2\rangle)=\langle 1\rangle$,
and therefore
$$
\mathrm{Aut}(G)=\langle\Delta_1\rangle\ltimes  (\mathrm{Inn}(G)\times\langle\Delta_2\rangle)=
(\langle\Delta_1\rangle\ltimes  \mathrm{Inn}(G))\times\langle\Delta_2\rangle.
$$

From $\mathrm{Aut}(G)=\langle\Delta_1\rangle (\mathrm{Inn}(G)\times\langle\Delta_2\rangle)$ we infer that 
$\mathrm{Inn}(G)\times\langle\Delta_2\rangle$ is the stabilizer of $A$ modulo~$T$ in $\mathrm{Aut}(G)$.
Since $A$ has infinite order modulo $T$, we see that $\Delta_1$ is not in this stabilizer.

\medskip

\noindent{\sc Step 16.} We have $(A\delta)^{\Delta_1}=A^{-1}\delta$ and $(B\delta)^{\Delta_1}=B^{-1}\delta$.

\medskip

This follows from the fact that $\Delta_1$ inverts $A$ and $B$.

\medskip

\noindent{\sc Step 17.} The stated matrix description of $\mathrm{Aut}(G)$ is correct. 

\medskip

For $i\neq j$, let $e_{i,j}\in M_n(\Z/(\b-1)\Z)$ be the matrix whose entries are all 0, except for an entry~1
in position $(i,j)$, and let $t_{i,j}=I+e_{i,j}$. Take $n=4$ and set $x=t_{(1,2)}$, 
$y=t_{(2,3)}$, $u=t_{(1,4)}$, and $d=\mathrm{diag}(1,-1,1,1)$.

Suppose first $\b$ is even. Then the assignment
$$
\Delta_1\mapsto d, A\delta\mapsto x,  B\delta\mapsto y,  \Delta_2\mapsto u,
$$
extends to an imbedding $\mathrm{Aut}(G)\hookrightarrow \GL_4(\Z/(\b-1)\Z)$ with the stated image.

Suppose next $\b$ is odd and $\b\notin\{-1,3\}$, and set $e=dy$. Then $(xy)^e=(xy)^{-1}$ and $y^e=y^{-1}$. Thus
the assignment
$$
\Delta_1\mapsto e, A\delta\mapsto xy,  B\delta\mapsto y,  \Delta_2\mapsto u,
$$
extends to an imbedding $\mathrm{Aut}(G)\hookrightarrow \GL_4(\Z/(\b-1)\Z)$ with the stated image, since
$\langle e, xy, y,u\rangle=\langle d, x, y,u\rangle$.
\end{proof}

Suppose finally that $\b\in\{-1,3\}$. In this case, $\mathrm{Inn}(G)\cong H(\Z/2\Z)\cong D_8$, while
$\Delta_1$ conjugates $A\delta$ into its inverse and commutes with $B\delta$. Take $n=5$ and the same choices made above for
the case when $\b$ is odd, except that now $d=t_{5,3}$ and
$e=dy=t_{5,3}t_{2,3}$. This yields an imbedding $\mathrm{Aut}(G)\hookrightarrow \GL_5(\Z/2\Z)$.
An alternative imbedding is given below.

\section{Studying the restriction map $\mathrm{Aut}(G)\to \mathrm{Aut}(\langle B,C\rangle$)}

Consider the restriction map $\Lambda:\mathrm{Aut}(G)\to \mathrm{Aut}(\langle B,C\rangle)$. By definition,
 $\Delta_2\in\ker(\Lambda)$.

\begin{prop}\label{inj} The restriction of $\Lambda$ to $\mathrm{Inn}(G)$ is injective.
\end{prop}

\begin{proof} By Proposition \ref{centro}, the centralizer of $\langle B,C\rangle$ in $G$ is $Z$.
\end{proof}

\begin{prop} Suppose that $\b\notin\{-1,3\}$. Then $\ker(\Lambda)=\langle \Delta_2\rangle$ and 
$\mathrm{Im}(\Lambda)\cong \langle\Delta_1\rangle\ltimes\mathrm{Inn}(G)$.
\end{prop}

\begin{proof} Let $\Gamma\in\ker(\Lambda)$.
We have $A^\Gamma=A^{\pm 1} B^i C^j$. Since $[A,C]=1$ is preserved by $\Gamma$, it follows
that $B^i$ commutes with $C$, and therefore $i\equiv 0\mod \b-1$. Thus, for a suitable choice
of $k$, we have $\Omega=\Gamma\Delta_2^k\in\ker(\Lambda)$ as well as 
$A^\Omega=A^{\pm 1} C^j$. Since $[A,B]=C$ is preserved by $\Omega$, it follows
that $[A^{\pm 1}C^j, B]= C$. Using the commutator identity $[xy,z]=[x,z]^y\; [y,z]$, we deduce that 
$$C=[A^{\pm 1},B]^{C^j}\, [C^j, B]=(C^{\pm 1})^{C^j}[C^j, B]=C^{\pm 1}[C^j, B].
$$
If the sign $+$ prevails, we deduce $[C^j,B]=1$, which forces $j\equiv 0\mod \b-1$.
Here $C^{\b-1}\in \langle B^{\b-1}\rangle$, so in this case $\Omega\in\langle \Delta_2\rangle$, whence $\Gamma\in \langle \Delta_2\rangle$.
Suppose, if possible, that the sign $-$ prevails.
Then $C^{-1} [C^j, B]=C$ forces $C^2\in\langle B\rangle$, this can only happen if $(\b-1)\mid 2$, that is, $\b\in\{-1,3\}$,
against our hypothesis. This shows that $\ker(\Lambda)=\langle \Delta_2\rangle$. The last statement follows
from Theorem \ref{auto}.
\end{proof}

\begin{prop} Suppose that $\b\in\{-1,3\}$. Then $T\cong Q_8$; 
$\ker(\Lambda)=\langle \Delta_2, \Delta_1\circ (BC)\delta\rangle\cong C_2\times C_2$; 
$\mathrm{Aut}(G)=\ker(\Lambda)\times \mathrm{Inn}(G)$; $\mathrm{Inn}(\Lambda)\cong H(\Z/2\Z)\cong D_8$;
and the assignment
$$
A\delta\mapsto t_{1,2}, B\delta\mapsto t_{2,3}, \Delta_2\mapsto t_{1,4}, \Delta_1\circ (BC)\delta\mapsto t_{1,5}
$$
extends to an imbedding $\mathrm{Aut}(G)\hookrightarrow \GL_5(\Z/2\Z)$.
\end{prop}

\begin{proof} That $T\cong Q_8$ was shown in Proposition \ref{str2}. Since $\mathrm{Aut}(Q_8)\cong S_4$
and the restriction of $\Lambda$ to $\mathrm{Inn}(G)\cong H(\Z/2\Z)\cong D_8$ is injective by Proposition \ref{inj},
we see that the image of $\Lambda$ is a Sylow 2-subgroup of $\mathrm{Aut}(Q_8)$. Since $\mathrm{Aut}(G)$
is a 2-group of order 32 by Theorem \ref{auto}, we have $|\ker(\Lambda)|=4$. We already know that $\Delta_2\in\ker(\Lambda)$
and we see that $(BC)\delta$ has order 2, it agrees with $\Delta_1$ on $\langle B,C\rangle$, and
it commutes with $\Delta_1$ and $\Delta_2$. Thus, $\ker(\Lambda)=\langle \Delta_2, \Delta_1\circ (BC)\delta\rangle\cong C_2\times C_2$.
Clearly $\ker(\Lambda)$ and $\mathrm{Inn}(G)$ are normal subgroups of $\mathrm{Aut}(G)$. By Theorem \ref{auto}, they intersect trivially,
so $\mathrm{Aut}(G)=\ker(\Lambda)\times \mathrm{Inn}(G)$, so the given assignment extends to the stated imbedding.
\end{proof}

Suppose $\b\in\{-1,3\}$, take $n=4$, and consider the assignment
$A\delta\mapsto t_{1,2}, B\delta\mapsto t_{2,3}$, and the corresponding imbedding $\mathrm{Inn}(G)\to \GL_4(\Z/2\Z)$, with image $H$.
The centralizer of $H$ in $\GL_4(\Z/2\Z)$ is isomorphic to $D_8$, whereas the centralizer of $\mathrm{Inn}(G)$ in $\mathrm{Aut}(G)$
is isomorphic to $C_2^3$. Thus, there is no room in $\GL_4(\Z/2\Z)$ to extend the given imbedding to all of $\mathrm{Aut}(G)$.

\medskip

Suppose next $\b=1-p^m$, where $p\in\N$ is an odd prime. Then $T=\langle B,C\rangle$ has defining relations
$$
B^{p^{2m}}=1=C^{p^m}, B^C=B^{1+p^m}.
$$
According to \cite[Section 4]{BC}, setting $x=B$ and $y=C$, we have $\mathrm{Aut}(T)=\langle a,b,c,d\rangle$, where
\begin{equation}\label{bc}
x^a=x^i, y^a=y; x^b=x, y^b=x^{p^m}y; x^c=xy, y^c=y; x^d=x, y^d=y^{1+p^m},
\end{equation}
and $[i]$ generates the group of units $[\Z/p^m\Z]^\times$ of $\Z/p^m\Z$. 

On the other hand, $A^{-1}\delta$ satisfies $B\mapsto BC$, $C\mapsto C$, so $(A^{-1}\delta)^\Lambda=c$;
$B^{-1}\delta$ satisfies $B\mapsto B$, $C\mapsto B^{p^m} C$, so $(B^{-1}\delta)^\Lambda=b$;
$\Delta_1$ satisfies $B\mapsto B^{-1}$, $C\mapsto B^{p^m} C$, so $\Delta_1^{\Lambda}=a^{p^{m-1}(p-1)/2} b$.
Thus, the image of $\Lambda$ is generated by $b$, $c$, and $a^{p^{m-1}(p-1)/2}$.
In other words, of all automorphisms of $T$, as described in \cite{BC}, the only ones that extend to an automorphism of~$G$
are those belonging to the subgroup generated by $b$, $c$, and $a^{p^{m-1}(p-1)/2}$, and this subgroup is isomorphic
to $C_2\ltimes H(\Z/p^m\Z)$. 

More generally, if $\b$ is even, $T$ is the direct product of its Sylow subgroups $T_p$
as $p$ runs through all positive odd prime factors of $\b-1$. An automorphisms $\Omega$ of $T$
is completely determined by its restrictions $\Omega_p$
to each $T_p$, and $\Omega$ extends to an automorphism of $G$ if and only if so does every $\Omega_p$.
In other words, $\mathrm{Aut}(T)$ is essentially described in \cite{BC}, and to identify among these
automorphisms those that extend to an automorphism of $G$, we can restrict attention to 
each $\mathrm{Aut}(T_p)$. We have $\b=1-p^m r$, where $m\geq 1$ and $p\nmid r$, so
we can find $s\in\N$ such that $rs\equiv 1\mod p^m$. There is a canonical projection $\pi_p:T\to T_p$,
and setting $B_p=B^{\pi_p}$ and $C_p=C^{\pi_p}$, we have $B_p^{C_p^s}=B_p^{1+p^m}$, so  setting $x_p=B_p$ and $y_p=C_p^s$, we have the automorphisms $a_p,b_p,c_p,d_p$ of $T_p$ form \cite{BC} defined by (\ref{bc}).
In this case, $c_p=(A^{-s}\delta)^{\Lambda_p}$, $b_p=(B^{-r}\delta)^{\Lambda_p}$,
and $\Delta_1^{\Lambda_p}=a_p^{p^{m-1}(p-1)/2} b_p$, so the image of $\Lambda_p:\mathrm{Aut}(G)\to \mathrm{Aut}(T_p)$
is $\langle b_p,c_p, a_p^{p^{m-1}(p-1)/2}\rangle$.


\section{On the isomorphism between $G(\b)_p$ and $G(\gamma)_p$}

Let $\gamma$ be an integer, $p$ a prime number,
and set $m=v_p(\b-1)$ and $n=v_p(\gamma-1)$. 
Let $T(\b)_p$ (resp. $T(\gamma)_p$) be the Sylow $p$-subgroup of $T(\b)$ (resp. $T(\gamma)$),
and set $G(\b)_p=\langle A\rangle\ltimes T(\b)_p$ (resp. $G(\gamma)_p=\langle A\rangle\ltimes T(\gamma)_p$).

Now $T(\b)=T(\b)_p\times T(\b)_{p\prime}$, where $T(\b)_{p\prime}$ is the product of the remaining Sylow subgroups of~$T(\b)$, 
and both factors are normal in $G$.
This implies that the map $\pi:G(\b)\to G(\b)_p$, given by $u t\mapsto u t_p$, where 
$u\in \langle A\rangle$, $t\in T(\b)$, and $t_p$ is the $T(\b)_p$-part
of $t$, is a group epimorphism with kernel $T(\b)_{p\prime}$, and we set $a=A^\pi=A$, $b=B^\pi$, and $c=C^\pi$. Let
$$
G(\gamma)=\langle X,Y\,|\, X^{[X,Y]}=X,\, Y^{[Y,X]}=Y^\gamma\rangle,
$$
which has a corresponding projection $\rho:G(\gamma)\to G(\gamma)_p$, and we set $x=X^\rho=X$, $y=Y^\rho$.

\begin{prop} We have $G(\b)_p\cong G(\gamma)_p$ if and only if $m=n$.
\end{prop}

\begin{proof} Assume that $G(\b)_p\cong G(\gamma)_p$. Then their torsion subgroups $T(\b)_p$ and $T(\gamma)_p$ are isomorphic. In particular,
$p^{3n}=|T(\gamma)_p|=|T(\b)_p|=p^{3m}$, and therefore $n=m$. Suppose, conversely, that $m=n$. 
We have $\b=1+p^m\ell$ and $\gamma=1+p^m r$, where $p\nmid\ell$ and $p\nmid r$. We look for $i\in\N$ such
that $\b^i\equiv \gamma\mod p^{2m}$, that is $1+i(\b-1)\equiv 1+(\gamma-1)\mod p^{2m}$, or $i\ell p^m\equiv r p^m\mod p^{2m}$, which means
$i\ell\equiv r\mod p^m$. This is certainly possible since $p\nmid\ell$.  Fix this $i$. Note
that if $j\in\N$ is inverse of $i$ modulo $p^{2m}$, then $jr\equiv \ell\mod p^m$. 

As $a$ and $b$ commute with $c$ modulo $Z(G(\b)_p)$,
it follows that $[a,b^i]\equiv i\mod Z(G(\b)_p)$, so
$$
{}^{[a,b^i]} b={}^{c^i} b=b^{\b^i}=b^\gamma.
$$
It follows the assignment $X\mapsto a$, $Y\mapsto b^i$ extends to a group epimorphism $f:G(\gamma)\to G(\b)_p$. As $f$
sends the torsion subgroup $T(\gamma)$ of $G(\gamma)$ into the torsion subgroup of $G(\b)_p$, namely the $p$-group $T(\b)_p$,
all Sylow $q$-subgroups, $q\neq p$, are in the kernel of $f$. This produces a group epimorphism
from the factor group $G(\gamma)_p$ of $G(\gamma)$ onto $G(\b)_p$, sending $x$ to $a$ and $y$ to $b^i$.
Likewise, the assignment
$A\mapsto x$, $B\mapsto y^j$ yields a group epimorphism $G(\b)_p$ onto $G(\gamma)_p$, sending $a$ to $x$ and $b$ to $y^j$.
Since $ij\equiv 1\mod p^{2m}$, these are inverse epimorphisms.
\end{proof}

\section{The automorphism group of $G(\b)/\langle A^{\b-1}\rangle$}

We assume $\gcd(\b-1,6)=1$ throughout this section and set $K=\langle A^{\b-1}\rangle$,
a central subgroup of $G$, as well as $L=G/K$, and write $\rho:G\to L$ for the natural projection,
with $a=A^\rho$, $b=B^\rho$, and $c=C^\rho$. Then $L=\langle a\rangle\ltimes (\langle c\rangle\ltimes
\langle b\rangle)\cong\mathrm{Hol}(\langle a\delta\rangle,\langle c\rangle\ltimes
\langle b\rangle)$
has defining relations
$$
b^{(\b-1)^2}=1=c^{\b-1}, {}^c b=b^\b, a^{\b-1}=1, b^a=bc^{-1}, c^a=c.
$$
In this section, we determine the automorphism group of $L$.

\begin{lemma} The subgroup $K$ is characteristic in $G$.
\end{lemma}

\begin{proof} Since $K\subset Z$, we have $(A^{\b-1})^g=A^{\b-1}$ for every $g\in G$.
Moreover, $A^\Delta_2=A B^{\b-1}$, with $B^{\b-1}\in Z$ and $B^{(\b-1)^2}=1$, so $(A^{\b-1})^\Delta_2=(A B^{\b-1})^{\b-1}=A^{\b-1}$.
Finally, $A^\Delta_1=A^{-1}$, so $\Delta_1$ inverts~$A^{\b-1}$. Now apply Theorem \ref{auto}.
\end{proof}

\begin{lemma}\label{r} The natural map $\tau:\mathrm{Aut}(G)\to \mathrm{Aut}(L)$ is injective.
\end{lemma}

\begin{proof} Let $\Omega\in\ker(\tau)$. Then $g^\Omega\equiv g\mod K$ for all $g\in G$.
Since $\langle B,C\rangle$ is characteristic in $G$, it follows that $B^\Omega=B$. Moreover, $A^\Omega=A k$, with $k\in K$,
and also $A^\Omega=A^{\pm 1}t$, with $t\in T$. Therefore, either the sign $+$ prevails, in which case $A^\Omega=A$,
or else the sign $-$ prevails, in which case $(\b-1)\mid 2$, against our hypothesis.
\end{proof}

\begin{lemma}\label{2cen} We have $Z(L)=\langle b^{\b-1}\rangle$ and $Z_2(L)=\langle b^{\b-1},c\rangle$.
\end{lemma}

\begin{proof} The center of $L$ is formed  by all $g^\pi$ such that $[g,x]\in K$ for all $x\in G$. But
$[g,x]$ is in $\gamma_2(G)=\langle B^{\b-1},C\rangle$, so $Z(L)=\{g^\pi\,|\, [g,x]=1\text{ for all }x\in G\}=Z(G)^\pi=\langle b^{\b-1}\rangle$.
Since $L/Z(G)\cong H(\Z/(\b-1)\Z)$, it follows that $Z_2(L)=\langle b^{\b-1},c\rangle$.
\end{proof}

\begin{lemma} Suppose $ij\equiv 1\mod (\b-1)^2$. Then the assignment $a\mapsto a^i$,
$b\mapsto b^j$ extends to an automorphism $\mu_i$ of $L$. The map $i\mapsto \mu_i$ yields an imbedding 
$[\Z/(\b-1)\Z]^\times\to \mathrm{Aut}(L)$.
\end{lemma}

\begin{proof} Since $a$ and $b$ commute with $c$ modulo $Z(L)$, it follows that $[a^i,b^j]\equiv c^{ij}\equiv c\mod Z(L)$.
\end{proof}

\begin{prop} The assignment $a\mapsto a$,
$b\mapsto ab$ extends to an automorphism $\Psi$ of $L$.
\end{prop}

\begin{proof} Notice that $[a,ab]=c$, so   ${}^{[a,ab]} (ab)= a b^\b$. Thus, it suffices to show $(ab)^{\b-1}=b^{\b-1}$, 
for in that case $(ab)^{\b}=ab^{\b}$
and $(ab)^{(\b-1)^2}=1$. Now
$$
(ab)^{\b-1}=a^{\b-1} b^{a^{\b-2}}\cdots b^{a} b=a^{\b-1} bc^{-(\b-2)}\cdots bc^{-1} b=a^{\b-1}c^{-\binom{\b-1}{2}} 
b^{\b^{\binom{\b-1}{2}}}\cdots b^{\b^{\binom{3}{2}}} b^{\b^{\binom{2}{2}}}b.
$$
Here $c^{-\binom{\b-1}{2}}=1=a^{\b-1}$, and $\b^j\equiv 1+j(\b-1)\mod (\b-1)^2$, $j\geq 0$, so
$$
{\b^{\binom{\b-1}{2}}}\cdots {\b^{\binom{3}{2}}} {\b^{\binom{2}{2}}}+1\equiv 
(\b-1)+(\b-1)\left[\binom{\b-1}{2}+\cdots+{\binom{3}{2}}+{\binom{2}{2}}\right].
$$
Now
$$2\left[\binom{\b-1}{2}+\cdots+{\binom{3}{2}}+{\binom{2}{2}}\right]=((\b-1)^2+\cdots+3^2+2^2)-(\b-1+\cdots+3+2),
$$
where the right hand side equals $(\b-1)\b(2(\b-1)+1)/6-(\b-1)\b/2$. Since $\gcd(\b-1,6)=1$, we see that indeed $(ab)^{\b-1}=b^{\b-1}$.
\end{proof}

We still write $\Delta_2$ for the automorphism of $L$ such that $a\mapsto ab^{\b-1}$ and $b\mapsto b$,
and note that, with a similar convention, $\Delta_1=\mu_{-1}$. We also set $M=\{\mu_i\,|\, [i]\in[\Z/(\b-1)\Z)]^\times\}$.

\begin{theorem} We have $\mathrm{Aut}(L)=\langle \Psi\rangle M (\mathrm{Inn}(L)\times \langle \Delta_2\rangle)$, with 
 $$\mathrm{Inn}(L)\times \langle \Delta_2\rangle=\mathrm{Aut}_2(L)\cong H(\Z/(\b-1)\Z)\times C_{\b-1},\;
\mathrm{Aut}(L)/\mathrm{Aut}_2(L)\cong \mathrm{Hol}(\Z/(\b-1)\Z).$$
\end{theorem}

\begin{proof} Let $\Gamma\in \mathrm{Aut}(L)$.

\medskip

\noindent{\sc Step 1.} We have $c^\Gamma=c^x b^y$, where $\gcd(x,\b-1)=1$ and $y\equiv 0\mod\b-1$.

\medskip

This follows from Lemma \ref{2cen} and the fact that $\Gamma$ induces automorphisms of   $Z_2(L)$ and $Z_2(L)/Z(L)$.

\medskip

\noindent{\sc Step 2.} We have $a^{\Gamma}=a^i b^j c^k$, where $j\equiv 0\mod\b-1$.

\medskip

From $[a,c]=1$, we infer $[a^i b^j c^k,c^x b^y]=1$, whence $[c^x,b^j]=1$, and therefore $[c,b^j]=1$.

\bigskip

\medskip

\noindent{\sc Step 3.} We have $a^{\Gamma_2}=a b^r c^s$, where $\Gamma_2=\Gamma\mu_q$ for a suitable $q$.

\medskip

Take $[q]$ to be the inverse of $[i]$ in $(\Z/(\b-1)\Z)^\times$, recalling that $c$ is fixed modulo $Z(L)$ by $\mu_q$.

\medskip

\noindent{\sc Step 4.} We have $a^{\Gamma_3}=a$, where $\Gamma_3=\Gamma_2(b\delta)^q\Delta_2^t$ for a suitable $q,t$.

\medskip

Taking into account Step 2, this follows as in Steps 4 and 7 of the proof of Theorem \ref{auto}.

\bigskip

As $L/Z_2(L)$ is a free $\Z/(\b-1)\Z$-module with basis formed by the cosest of $a$ and $b$, we have a group homomorphism
$\Omega:\mathrm{Aut}(L)\to \GL_2(\Z/(\b-1)\Z)$. By definition, $\ker(\Omega)=\mathrm{Aut}(L)$. Note that, in agreement with our convention on function composition, the matrix of
a linear transformation is constructed row by row instead of column by column.

\bigskip

\noindent{\sc Step 5.} We have $b^{\Gamma_3}=a^u b^v c^w$, where $\gcd(v,\b-1)=1$.

\medskip

The matrix $\Gamma^\Omega$ has first row $([1], 0)$, so its entry (2,2) must be invertible.

\medskip

\noindent{\sc Step 6.} We have $b^{\Gamma_4}=b^d c^e$, where $\Gamma_4=\Gamma_3\Psi^q$ for a suitable $q$.

\medskip

As $Z_3(L)=L$, we have $(a^q b)^v\equiv a^{qv} b^v\mod Z_2(L)$, with $Z_2(L)$ as in Lemma \ref{2cen}.
Since $[v]$ is invertible, the congruence $qv\equiv -u\mod \b-1$ is solvable. 

\medskip

\noindent{\sc Step 7.} We have $b^{\Gamma_4}=b^f$, $a^{\Gamma_4}=a$, and $c^{\Gamma_4}=b^g c$,
where $\Gamma_4=\Gamma_3 (a\delta)^q$ for a suitable $q$.

\medskip

This follows as in Steps 10 and 11 of the proof of Theorem \ref{auto}.

\medskip

\noindent{\sc Step 8.} We have $b^{\Gamma_5}=b$ and $a^{\Gamma_5}=a$, 
where $\Gamma_5=\Gamma_4 (c\delta)^q$ for a suitable $q$, so $\mathrm{Aut}(L)$ is generated by $\mathrm{Inn}(L)$
together with $\Delta_2$, $\Psi$, and $M$.

\medskip

This follows as in Steps 12 and 13 of the proof of Theorem \ref{auto}.

\medskip

\noindent{\sc Step 9.} The kernel of $\Omega$ is $\mathrm{Aut}_2(L)=\mathrm{Inn}(L)\times\langle \Delta_2\rangle$.

\medskip

Since $Z_3(L)=L$, it follows that $\mathrm{Inn}(L)$ is included in $\mathrm{Aut}_2(L)$. On the other hand,
we already know that $\Delta_2\in \mathrm{Aut}_1(L)$, where $\mathrm{Aut}_1(L)\subseteq\mathrm{Aut}_2(L)$. That $\mathrm{Inn}(L)$ and $\langle \Delta_2\rangle$ generate their direct product
follows from the corresponding result for $G$ by means of Lemma \ref{r}. That $\mathrm{Aut}_2(L)$
is included in $\mathrm{Inn}(L)\times\langle \Delta_2\rangle$ follows from the previous steps, as $\Psi$
and $M$ are not needed when $\Gamma\in \mathrm{Aut}_2(L)$. 

\medskip

\noindent{\sc Step 10.} We have $\mathrm{Aut}(L)=M\langle\Psi\rangle(\mathrm{Inn}(L)\times\langle \Delta_2\rangle)$
and $\mathrm{Aut}(L)/\mathrm{Aut}_2(L)\cong \mathrm{Hol}(\Z/(\b-1)\Z)$.

\medskip

By Step 9, $\mathrm{Inn}(L)\times\langle \Delta_2\rangle$ is normal in $\mathrm{Aut}(L)$, so 
$\langle\Psi\rangle(\mathrm{Inn}(L)\times\langle \Delta_2\rangle)$ is a subgroup of $\mathrm{Aut}(L)$.
Let $U_i$ and $V$ be the matrices corresponding to $\mu_i$ and $\Psi$ via $\Omega$, respectively. Then
$$
V^{U_i}=V^{i^2}.
$$
This shows that $M$ normalizes $\langle \Psi\rangle$ modulo $\ker(\Omega)=\mathrm{Aut}_2(L)$, so
$M\langle\Psi\rangle(\mathrm{Inn}(L)\times\langle \Delta_2\rangle)$ is also a subgroup, necessarily equal
to $\mathrm{Aut}_2(L)$. Finally, $\mathrm{Aut}(L)/\mathrm{Aut}_2(L)$ is generated by the matrices $U_i$ and $V$
described above, with $i$ running through $[\Z/(\b-1)\Z]^\times$. These matrices produce a copy of $\mathrm{Hol}(\Z/(\b-1)\Z)$, where
the usual action of $[\Z/(\b-1)\Z]^\times$ on $\Z/(\b-1)\Z$ by multiplication is preceeded by the automorphism ``squaring"
of $[\Z/(\b-1)\Z]^\times$. This operation never changes the isomorphism type of a semidirect product.

\end{proof}


\end{document}